\documentclass[11pt]{amsart}
\usepackage{epigamath-voisin}

\usepackage[english]{babel}

\numberwithin{equation}{section}

\usepackage{amsmath,amssymb,amscd,amsfonts}
\usepackage[T1]{fontenc}
\usepackage[utf8]{inputenc}
\usepackage{epic}
\usepackage{dsfont}

\usepackage[shortlabels]{enumitem}
\setlist[enumerate,1]{label={\rm(\arabic*)}, ref={\rm\arabic*}} 

\newtheorem{theorem}{Theorem}[section]
\newtheorem{lemma}[theorem]{Lemma}
\newtheorem{claim}[theorem]{Claim}
\newtheorem{proposition}[theorem]{Proposition}
\newtheorem{corollary}[theorem]{Corollary}

\theoremstyle{definition}
\newtheorem{definition}[theorem]{Definition}

\theoremstyle{remark}
\newtheorem{remark}[theorem]{Remark}
\newtheorem{example}[theorem]{Example}
\newtheorem{question}[theorem]{Question}
\newtheorem{problem}[theorem]{Problem}

\def \Q{{\mathbb Q}}
\def \P{{\mathbb P}}
\def \Z{{\mathbb Z}}
\def \C{{\mathbb C}}
\def \R{{\mathbb R}}

\DeclareMathOperator{\codim}{codim}
\DeclareMathOperator{\Torsion}{Torsion}
\DeclareMathOperator{\prim}{prim}
\DeclareMathOperator{\SL}{SL}
\DeclareMathOperator{\Hg}{Hg}
\DeclareMathOperator{\Sp}{Sp}
\DeclareMathOperator{\Aut}{Aut}

\YearArticle{2023} \EpigaArticleNr{11} \ReceivedOn{December 16, 2022}
\InFinalFormOn{August 16, 2023}
\AcceptedOn{September 11, 2023}

\title{On algebraically coisotropic submanifolds of holomorphic symplectic manifolds}
\titlemark{On algebraically coisotropic submanifolds of holomorphic symplectic manifolds}
\author{Ekaterina Amerik}
\address{Universit\'e Paris-Sud,  
Laboratoire de Math\'ematiques d'Orsay, 
91405 Orsay, France \\
and National Research University Higher School of Economics, 
Laboratory of Algebraic Geometry and its Applications, 
Usacheva 6, 119048 Moscow, Russia}
\email{ekaterina.amerik@math.u-psud.fr}
\author{Fr\'ed\'eric Campana}
\address{Universit\'e Lorraine, Institut \'Elie Cartan, 
54506 Vand{\oe}uvre-l\`es-Nancy, France}
\email{frederic.campana@univ-lorraine.fr}

\authormark{E.~Amerik and F.~Campana}

\AbstractInEnglish{We investigate algebraically coisotropic submanifolds $X$ in a holomorphic symplectic projective manifold $M$. Motivated by our results in the hypersurface case, we raise the following question: when $X$ is not uniruled, is it true that up to a finite \'etale cover, the pair $(X,M)$ is a product $(Z\times Y, N\times Y)$, where $N, Y$ are holomorphic symplectic and $Z\subset N$ is Lagrangian? We prove that this is indeed the case when $M$ is an Abelian variety and give a partial answer when the canonical bundle $K_X$ is semiample. In particular, when $K_X$ is nef and big, $X$ is Lagrangian in $M$ (using a recent text of Taji, we could also obtain this for $X$ of general type).  We also remark that Lagrangian submanifolds do not exist on a sufficiently general Abelian variety, in contrast to the case when $M$ is irreducible hyperk\"ahler.}



\acknowledgement{The first author was partially supported by the HSE University basic research program and by ANR Project FanoHK (ANR-20-CE40-0023).}

\dedication{To Claire Voisin, for her birthday volume}

\begin{document}

%%%%%%%%%%%%%%%%%%%%%%%%%%%%%%%
% Title page
%%%%%%%%%%%%%%%%%%%%%%%%%%%%%%%

\maketitle

\begin{prelims}

\DisplayAbstractInEnglish

\bigskip

\DisplayKeyWords

\medskip

\DisplayMSCclass

\end{prelims}

%%%%%%%%%%%%%%%%%%%%%
% Table of Contents
%%%%%%%%%%%%%%%%%%%%%

\newpage

\setcounter{tocdepth}{1}

\tableofcontents

%% general queries

%%%%%%%%%%%%%%%%%%%%%
% Content begins here
%%%%%%%%%%%%%%%%%%%%%

\section{Introduction and main results}

\subsection{Algebraically coisotropic submanifolds} 

Let $M$ be a complex projective manifold equipped with a holomorphic symplectic form $\sigma$. In particular, $M$ is of even dimension $2n$. Let $X$ be an irreducible complex submanifold of $M$. By linear algebra, at every $x\in X$, the corank of the restriction $\sigma|_X$ at $x$ is at most $c=\codim(X)$. Equality holds exactly when $T_{X,x}$ contains its $\sigma$-orthogonal.

\begin{definition}\label{coiso}
The submanifold  $X$ is \emph{coisotropic} if the corank of $\sigma|_X$ is equal to $c$ at each\footnote{This definition coincides with those of \cite{Saw} and \cite{Vo} for $X$ smooth. More generally, when $X$ is a possibly singular subvariety, one calls $X$ coisotropic when $T_{X,x}$ contains its $\sigma$-orthogonal for every smooth point $x\in X$.} point of $X$.  In this case, the kernel of $\sigma|_X$ defines a regular foliation\footnote{\textit{A priori} only a distribution, it is a foliation because $\sigma$ is $d$-closed. See \cite[Section~2.1]{Saw}.} $\mathcal{F}$ of rank $c$ on $X$, and the rank of $\sigma|_X$ is $2(n-c)$. The foliation $\mathcal{F}$ is called the \emph{characteristic foliation} of $X$.
\end{definition}

It follows from the definition that the codimension $c$ of a coisotropic subvariety $X$ is at most $n$, and if it is equal to $n$, then  the restriction of $\sigma$ to $X$ is zero. In this case, $X$ is called \emph{Lagrangian}.
Also recall  that a subvariety $Y$ is \emph{isotropic} if $\sigma|_Y=0$; the dimension of an isotropic subvariety is at most $n$, and to be Lagrangian is the same as to be isotropic and coisotropic at the same time.

In general, the set of the coisotropic submanifolds of $M$ depends on the choice of the form $\sigma$; when $\sigma$ is not specified, we say that a submanifold of $M$ is coisotropic when it is coisotropic for some $\sigma$ on $M$. One important and most studied case is when $M$ is \emph{irreducible holomorpic symplectic} (\emph{IHS}), or \emph{hyperk\"ahler}: that is, $M$ is simply connected and $h^{2,0}$ is $1$-dimensional.
In this case, obviously the notion of the coisotropy is independent of the form. When $M$ is a product of IHS manifolds, by the K\"unneth formula, any symplectic form is a product of forms lifted from the factors, and so the notion of coisotropy does not depend on the choices in the case when $X$ is itself a compatible product. 

A regular foliation on an algebraic manifold is said to be \emph{algebraic} (or \emph{algebraically integrable}) if its leaves are algebraic submanifolds. If the characteristic foliation on $X$ is algebraic, $X$ is said to be \emph{algebraically coisotropic}. In this case, there is a fibration $f\colon X\to B$ which has the leaves of $\mathcal{F}$ as fibres.
This fibration is ``quasi-smooth'' in the sense that all of its fibres have smooth reduction. We call it the \emph{characteristic fibration} of $X$.

\begin{example}
Any smooth hypersurface $D$ of $M$ is coisotropic since $\sigma|_D$ is of even rank and therefore degenerate.

Lagrangian subvarieties of $M$ are algebraically coisotropic: the characteristic foliation then has a single leaf, $X$ itself. In particular, curves on a holomorphic symplectic surface (either $K3$ or Abelian) 
are algebraically coisotropic.
\end{example}

\subsection{Divisors}

It is easy to see that a smooth uniruled hypersurface is algebraically coisotropic with rational curves as leaves. On the opposite side, Hwang and Viehweg \cite{HV} have shown that a smooth algebraically coisotropic hypersurface of general type is a curve in a holomorphically symplectic surface. Several important observations have been made in \cite{Saw}. In \cite{AC-div}, we have described the algebraically coisotropic divisors as follows.

\begin{theorem}\label{divisors}
A smooth algebraically coisotropic hypersurface $D$ in a holomorphic symplectic manifold $M$ is either uniruled or, up to a finite \'etale cover of\, $M$, of the  form $D=C\times Y\subset M=S\times Y$, where $Y$ is a holomorphic symplectic manifold, $S$ is a holomorphic symplectic surface and $C$ is a curve on $S$.
\end{theorem}

\subsection{Higher codimension}

It is natural to ask whether a similar description is still valid for algebraically coisotropic subvarieties of higher codimension.

%\medskip

The analogues of curves in holomorphically symplectic surfaces are Lagrangian subvarieties.  The most optimistic higher-codimensional analogue of Theorem~\ref{divisors} would be a positive answer to the following.

\begin{question}\label{q1}
Let $X\subset M$ be a non-uniruled algebraically coisotropic submanifold in a projective holomorphic symplectic manifold $M$. Up to a finite \'etale cover, do we have $X=L\times Y$ and $M=N\times Y$, where $Y$ and $N$ are holomorphically symplectic and $L\subset N$ is Lagrangian in $N$?

In particular, is $X$ Lagrangian if $K_X$ is big? Is $X$ Lagrangian if $M$ is either a simple Abelian variety or an irreducible hyperk\"ahler manifold?
\end{question} 

\begin{remark}\label{just-coisotrop}
While every hypersurface is coisotropic, most  higher-codimensional submanifolds are not; in fact, coisotropy is a cohomological condition (see \cite[Lemma~1.4]{Vo}). Notice however that not all coisotropic submanifolds are of the type described in Question~\ref{q1}; \textit{e.g.}\ when $M$ is irreducible hyperk\"ahler, a complete intersection of two hypersurfaces with zero Beauville--Bogomolov intersection is coisotropic, essentially by definition of the Beauville--Bogomolov form. In particular, when $M$ has a Lagrangian fibration $f\colon M\to \P^n$, the inverse image of a codimension $2$ complete intersection (and, more generally, of any submanifold) $Y\subset \P^n$ is coisotropic. See Proposition~\ref{renat} for a result on its characteristic foliation.
\end{remark}

\begin{remark}\label{r1}
A positive answer to Question~\ref{q1} would reduce the description of algebraically coisotropic $X\subset M$ to that of the Lagrangian submanifolds of symplectic manifolds. Up to a finite covering, these are products $M=\Pi_iM_i$, where the $M_i$ are either simple Abelian varieties or IHS manifolds. However, the symplectic form does not need to be a product of forms on the factors if there is a non-trivial torus part (see \cite{AC-div}).
\end{remark}

The aim of this note is to prove several very partial results on these questions. Concerning Lagrangian submanifolds, we provide a brief survey on IHS manifolds in
Section~\ref{SIHS} and make some remarks on Abelian varieties in Section~\ref{SAV}. 

\begin{theorem}\label{kappaf}
Let $X$ be an  algebraically coisotropic submanifold of a holomorphic symplectic manifold $M$. Let $f\colon X\to B$ be its characteristic fibration. Suppose that $K_X$ is semiample. Then $f$ is isotrivial, and $\kappa(X)=\kappa(F)$, where $F$ is a smooth fibre of $f$.
\end{theorem}

The proof of Theorem~\ref{kappaf} is a direct adaptation of the one given in \cite{AC-div} as the main step of the proof of Theorem~\ref{divisors}. 

Using \cite{T21}, one can weaken the semiampleness condition as follows.

\begin{theorem}\label{kappaf'}
Let $X$ be an  algebraically coisotropic submanifold of a holomorphic symplectic manifold $M$. Let $f\colon X\to B$ be its characteristic fibration. Suppose that the smooth fibres of $f$ have good minimal models. Then $f$ is isotrivial, and $\kappa(X)=\kappa(F)$, where $F$ is a smooth fibre of $f$.
\end{theorem} 

One may expect that the pseudo-effectivity of $K_X$ is sufficient for the conclusions of Theorem~\ref{kappaf'}, but new techniques are certainly required. See Remark~\ref{r2} for a related question.

\begin{corollary}\label{nefbig}
Let $M$ be a holomorphic symplectic manifold and $X\subset M$ an algebraically coisotropic submanifold. Assume that $K_X$ is nef and big, or more generally that $X$ is of general type. Then $X$ is Lagrangian.
\end{corollary}

Indeed, if $K_X$ is nef and big, then so is $K_F=K_X|_F$, and the Kawamata base-point-free theorem implies that $K_F$ is semiample. Hence $\kappa(X)=\kappa(F)$, and by the bigness of $K_X$, we must have $X=F$; that is, $X$ is Lagrangian.  If $X$ is of general type, so is $F$; in particular, it has a good minimal model, and we obtain the same conclusion from Theorem~\ref{kappaf'}.

Note that this is exactly the analogue of the Hwang--Viehweg theorem in higher codimension.

\begin{corollary}\label{absimple}
If\, $M$ is a simple Abelian variety and $X\subset M$ is algebraically coisotropic, then $X$ is Lagrangian.
\end{corollary}

Indeed, all positive-dimensional submanifolds in a simple Abelian variety have ample canonical bundle (see \cite[Proposition 4.1]{H}).

Examples of Lagrangian submanifolds in simple Abelian varieties of dimension $2n\geq 4$ seem difficult to construct.  O.~Debarre and C.~Voisin have informed us of one such construction for $n=2$, due to Schoen; 
see \cite{S, CMR}.  We do not know any higher-dimensional examples. They do not exist on sufficiently general (``Hodge-general'') Abelian varieties, by Corollary~\ref{chodge}.

The answer to Question~\ref{q1} is positive when $M$ is an Abelian variety.

\begin{theorem}\label{tab}
Let $M$ be an Abelian variety and $X\subset M$ an algebraically coisotropic submanifold. Then after a finite \'etale cover, there are subtori $D,N, C, P$ of\, $M$ such that $M=D\times C\times N\times P$ and a submanifold $Z\subset N$ such that 
\begin{enumerate}
\item $Z$ generates $N$;
\item $X=D\times C\times Z$;
\item the restriction $\sigma_N$ of\, $\sigma$ to $N$ is symplectic, and $Z$ is Lagrangian in $N$;
\item the characteristic fibration of\, $X$ is the projection $f\colon X\to D$ with fibres $C\times Z$, and the restriction $\sigma_D$ is symplectic;
\item $\dim(C)=\dim(P)$; in particular, if\, $X$ generates $M$, then $C=0$ and moreover $\sigma=\sigma_D\oplus\sigma_N$.
\end{enumerate}
Conversely, given the tori $D,C,N,P$, the symplectic forms $\sigma_N$ and $\sigma_D$, and a $\sigma_N$-Lagrangian submanifold of general type $Z$ generating $N$, then $X=D\times C\times Z$ is algebraically coisotropic in $D\times C\times N\times P$ for a suitable $\sigma$ which restricts as $\sigma_N$, resp.\ $\sigma_D$, to $N$, resp.\ $D$.
\end{theorem}

\begin{remark}
We shall describe all such $\sigma$ in the proof of Theorem~\ref{tab}.
\end{remark}

\subsection*{Acknowledgments}
We are grateful to E.~Macri and Y.~Zarhin for useful discussions, and to O.~Debarre and C.~Voisin for indicating Schoen's examples of simple Abelian fourfolds with Lagrangian surfaces. 

\section{Proofs of Theorems~\ref{kappaf} and~\ref{kappaf'}}

The proofs of these theorems use several auxiliary results.  First of all, let us recall some general background. We follow \cite{AC-div}, where we considered only algebraically coisotropic hypersurfaces, but many starting remarks are valid in general.

Let $X$ be algebraically coisotropic; then the family of leaves of the characteristic foliation ${\mathcal F}$ defines a morphism from $X$ to its Chow variety. Normalizing the image if necessary, we obtain a fibration $f\colon X\to B$ with fibres which are leaves of ${\mathcal F}$ (the ``characteristic fibration''). We recall that the codimension $c$ of $X$ is equal to the relative dimension of $f$, and we usually denote by $F$ the general fibre of $f$. Each (set-theoretic) fibre is smooth, but some of them can be multiple as $f$ is not necessarily a smooth morphism. Equivalently, some leaves of $f$ have non-trivial holonomy. The base $B$ can be singular at the corresponding points but has only quotient singularities. In fact, all holonomy groups of the foliation are finite, and locally in the neighbourhood of the fibre $F_0$, the variety $X$ is diffeomorphic to the quotient of $T\times F$ by the holonomy group $G$, where $F$ is the Galois covering of $F_0$ with Galois group $G$, $T$ is a local transverse and $G$ acts diagonally. This result is known in foliation theory as \emph{Reeb stability}. In the neighbourhood of $F_0$, the map $f$ is the projection to $T/G$.

Complex-analytically, it is not always true that a neighbourhood of $F_0$ in $X$ is a quotient of the product because the complex structure on the general leaves $F$ can vary, but it remains true that the quotient $T/G$ gives a local model for $B$.

The following lemma is due to Sawon (see \cite[Lemma~6]{Saw} for the case of hypersurfaces). The argument is purely local on $X$, depending only on the $d$-closedness of $\sigma$.

\begin{lemma}\label{base}
Let $X$ be smooth and algebraically coisotropic, and let $f\colon X\to B$ be the characteristic fibration. Denote by $f\colon X^0\to B^0$ the restriction to the smooth locus $($that is, $B^0=B-E$, where $b\in E$ if and only if the fibre $X_b$ is a multiple fibre\,$)$.
Then there is a holomorphic symplectic form $\eta$ on $B^0$ with $f^*\eta=\sigma|_X$.
\end{lemma}

\begin{proof}
  As in \cite{Saw}, we remark that the restriction of $\sigma$ gives a holomorphic symplectic form on any local transverse and that if $U_1$ and $U_2$ are two local transverses over the same small $U\subset B^0$, then the natural isomorphism $\phi\colon U_1\to U_2$ (sending the intersection point of $U_1$ with some leaf $L$ to that of $U_2$ with $L$) can be viewed as a $t=1$ map of a flow $\phi_t$ preserving $\sigma|_X$. Using the $d$-closedness of $\sigma$, one obtains as in \cite{Saw} that the forms agree on local models: $\sigma|_{U_1}=\phi^*\sigma|_{U_2}$, yielding a global holomorphic symplectic form $\eta$ on $B^0$.
  \end{proof}

%\medskip

The next result has been obtained for hypersurfaces in \cite[Lemma~2.3]{AC-div}, and its proof in the general case is analogous.

\begin{proposition}\label{nomult}
In the situation of Lemma~\ref{base}, the fibration $f$ has no multiple fibres in codimension $1$, so that $\sigma$ descends to a symplectic form $\eta$ outside of a subset of codimension at least $2$ in $B$. Moreover, the highest exterior power $\alpha=\eta^{n-c}$ trivializes $K_B$, and $B$ has only canonical singularities. Therefore:
\begin{enumerate}
\item Any smooth model of\, $B$ has $\kappa=0$, and in particular $B$ is not uniruled.
\item The smooth locus $B^0$ of\, $B$ is special in the sense of\, \cite{AC-iso}.
\end{enumerate}
\end{proposition}

\begin{proof}
The argument for the absence of multiple fibres in codimension $1$ is that of \cite[Lemma~2.3]{AC-div}; we only adapt the notation. 
We set $\omega=\sigma^{n-c}$, so that $f^*\alpha=\omega$ over $B^0$. Suppose there is a divisor of multiple fibres; then, replacing
$B$ with a small open neighbourhood of a general point of the image $E$ of this divisor, we may assume that $B$ is a polydisk with coordinates $(u_1, \dots, u_{2(n-c)})=(u,u')$ and $E$ is given by $u=0$ (here $u'$ denotes the $(2(n-c)-1)$-tuple of remaining coordinates) . Moreover, we can choose local coordinates $(z_1,\dots, z_{2(n-c)}, w_{2(n-c)+1},\dots, w_{2n-c})=(z, z', w)$ on $X$ such that $f$ is locally given by $u=z^m$, $u'=z'$. Then the same calculation as in \cite[Lemma~2.3]{AC-div} shows that $m=1$.

Hence $\alpha$ is a non-vanishing holomorphic $2(n-c)$-form defined over a complement of a subset of codimension at least $2$, and by definition it trivializes $K_B$. Moreover, since $B$ has only quotient singularities, a lemma due to Freitag \cite{F} shows that $\alpha$ extends to any resolution $r\colon B'\to B$ of $B$, and so the singularities of $B$ are canonical; hence $\kappa(B')=0$. If $E$ is the reduced exceptional set of $r$, then $h^0(B', m(K_{B'}+E))=h^0(B',mE')=1$ for each $m>0$ since $K_{B'}+E:=E'\geq E$ is an effective divisor supported on $E$. Thus $\kappa(B',K_{B'}+E)=0$, and $(B',E)$ is special.
\end{proof}

%\medskip

Also notice that in the exact sequence $0\to {\mathcal F}\to T_X\to N_F\to 0$, the normal bundle $N_F$ is equipped with a non-degenerate $2$-form; hence it has trivial determinant.  It follows that $K_{\mathcal F}=K_X$.

\begin{proof}[Proof of Theorem~\ref{kappaf}]
When $K_F$ is semiample, the results Theorem 5.1 and Corollary 5.2 of \cite{AC-iso} apply, yielding that $f$ is an isotrivial fibration. Moreover, by Proposition~\ref{nomult}, $f$ has no multiple fibre in codimension~$1$.  Therefore, $K_X=K_{X/B}+f^*K_B$. By the same proposition, $K_B$ is trivial and so $K_X=K_{X/B}$.

By isotriviality, the family $f\colon X\to B$ trivializes after a finite covering $B'\to B$: the base change $f'\colon X'\to B'$ is the projection from $X'=F\times B'$ to the second factor.  Let $g$ be the natural base-change morphism from $X'$ to $X$ and $h$ the projection from $X'$ to $F$. Then $h^*(K_F)= K_{X'/B'}= g^*K_{X/B}$. Hence $\kappa(F)$ equals the Iitaka dimension of $K_{X/B}$, which is equal to $\kappa(X)$.
\end{proof}

\begin{proof}[Proof of Theorem~\ref{kappaf'}]
We repeat the preceding arguments, replacing \cite[Theorem 5.1, Corollary 5.2]{AC-iso} with \cite[Theorem 1.1]{T21}, applied to the special base $B^0$.
\end{proof}

\begin{proof}[Proof of Corollary~\ref{absimple}]
The structure of a submanifold of an Abelian variety is described in \cite{Ueno}: if $X$ is such a submanifold, then there exists an Abelian subvariety $A\subset M$ such that $X=p^{-1}(Z)$, where $p$ is the projection from $M$ to the quotient Abelian variety $L=M/A$ and $Z\subset L$ is a subvariety of general type.
 Now if $M$ is simple, then $A$ is trivial and $Z=X$, so that Corollary~\ref{nefbig} applies (note that $K_Z$ is nef and big, and even ample when $M$ is simple).
\end{proof}

\section{Proof of Theorem~\ref{tab}}
Let $M$ be an Abelian variety, and let $X$ be an algebraically coisotropic submanifold of $M$. We derive from Ueno's classification result \cite[Theorem 10.9]{Ueno} that $K_X$ is semiample: indeed, in the notation right above, $K_X$ is the inverse image of $K_Z$, which is nef and big, hence semiample by Kawamata base-point-freeness. Clearly, in the same notation, the restriction $g\colon  X\to Z$ of the projection $p\colon  M\to L=M/A$ is the Iitaka fibration of $X$, the fibres of $g$ are isomorphic to $A$, and $\kappa(X)=\kappa(Z)=\dim(Z)$. Denote by $N$ the subtorus of $L$ generated by $Z$.

Since $K_X$ is semiample, Theorem~\ref{kappaf} applies. Let $f\colon X\to B$ be the characteristic fibration on $X$; it is isotrivial with general fibre $F$. Since $\kappa(X)=\kappa(F)$, the restriction of $g$ to $F$ is the Iitaka fibration of $F$. Applying Ueno's theorem again, we see that the fibre of $g|_F$ is a subtorus $C$ of $A$.

Now Poincar\'e complete reducibility yields that passing to a finite \'etale covering, we may assume $A=D\times C$, $L=N\times P$ and $M=D\times C\times N \times P$, so that $Z\subset N$ induces the embedding of $X=D\times C\times Z$ into $M$, and $g\colon X=D\times C\times Z\to Z$ and $f\colon D\times C\times Z\to D$ are natural projections.

We shall now analyze the form $\sigma$ according to the decomposition $M=D\times C\times N\times P$. We view any torus as the quotient of its tangent space by a lattice, and by abuse of notation, denote by the same letter a subtorus of $M$ and its tangent space at any point. By the coisotropy condition, $C+T_{Z,z}$ is orthogonal (with respect to $\sigma$) to $C+D$ at each point $z\in Z$. Thus $C$ and $D$ are orthogonal, and $C+D$ is orthogonal to $T_{Z,z}$ for all $z\in Z$. Since $Z$ generates the subtorus $N$, the spaces $T_{Z,z}\subset N$ generate the vector space $N$ (that is, $T_N$), by the Gauss map. Thus $C+N$ and $C+D$ are orthogonal. Equivalently, $C$ is isotropic, and $C$, $N$ and $D$ are pairwise orthogonal.

The K\"unneth formula then implies that $\sigma=\sigma_D\oplus\sigma_N\oplus \sigma_P\oplus s$, where $\sigma_N$ is the restriction of $\sigma$ to $N$, and similarly for $D$ and $P$, and $s$ comes from an element in $H^{1,0}(P)\otimes H^{1,0}(D\times C\times N)$. This means the following: we have a pairing $s(p,q)=\sum_i\alpha_i(p)\beta_i(q)$ for $p\in P$, $q \in Q:=D\times C\times N$ and linear forms $\alpha_i$, $\beta_i$, and we associate to it the alternating bilinear form $s'$ on $P\times Q$ defined by $s'(p+q,p'+q'):=s(p,q')-s(p',q)$. 

Observe that the restriction $\sigma_D$ of $\sigma$ to $D$ is a symplectic form, for example by Lemma~\ref{base}. The following claim is crucial. 

\begin{claim}\label{restriction}
The restriction $\sigma_N$ of\, $\sigma$ to $N$ is a symplectic form, and $Z$ is coisotropic in $N$.
\end{claim}

\begin{proof}
Indeed, let $K$ be the kernel of $\sigma_N$. It gives a foliation ${\mathcal K}$ on $N$, which is a trivial subbundle in the tangent bundle $T_N$. The orthogonality relations established above imply that $K$ is contained in the $\sigma$-orthogonal to $D\times C\times N$, hence also in the $\sigma$-orthogonal to $T_X$ at any point $x\in X$. So $K$ is also contained in the $\sigma_N$-orthogonal to $T_Z$ at any point $z\in Z$. In other words, ${\mathcal K}|_Z\subset T_Z$ (that is, $Z$ is invariant by the foliation ${\mathcal K}$). But since $Z$ is of general type, it has only finitely many automorphisms. So $H^0(Z, T_Z)=0$ and $T_Z$ cannot have a trivial subbundle of positive rank; hence $K=0$ and $\sigma_N$ is symplectic.

To deduce the second part of the claim, consider the restriction $\sigma '$ of $\sigma$ to $D\times C\times N$. Since the $\sigma '$-orthogonal to $T_X$ is contained in its $\sigma$-orthogonal, $X$ is also coisotropic with respect to $\sigma'$ (by definition, we say that $X$ is coisotropic with respect to an arbitrary, not necessarily non-degenerate, form $\sigma'$ if $T_X$ contains its $\sigma'$-orthogonal). By the orthogonality relations above, this implies that $Z$ is $\sigma_N$-coisotropic in $N$, finishing the proof of the claim.
\end{proof}

We now give two different proofs of the fact that $Z\subset N$ is Lagrangian and $\dim(C)=\dim(P)$. The first one is shorter, the second one additionally gives information on $\sigma$.

\smallskip
{\it First proof.} 
We start by recalling that $X$ is $\sigma$-coisotropic in $M$ and note that $Q=D\times C\times N$ is $\sigma$-coisotropic in $M$ as well: indeed, since $X\subset Q$, the $\sigma$-orthogonal to $Q$ is contained in the $\sigma$-orthogonal to $T_X$ at every point of $X$, so it is also contained in $Q$ at every point of $X$. By the triviality of the tangent bundle of a torus, the orthogonal to a subtorus at a point does not depend on this point. Hence 
the $\sigma$-orthogonal to $Q$ is contained in $Q$ at every point, meaning that $Q$ is coisotropic. 

By the definition of coisotropy, the codimension of $Q=D\times C\times N$ in $M$ is equal to the corank of the restriction of $\sigma$ to $D\times C\times N$; that is, $\dim(P)=\dim(C)$. In the same way, the coisotropy of $X$ in $M$ gives $\dim(C\times Z)=\codim(X,M)$ (indeed, the fibre of the coisotropic fibration is $C\times Z$), hence $\dim(C)+\dim(Z)=\codim(Z,N)+\dim(P)$ and so again $\dim(Z)=\codim(Z,N)$, meaning that $Z$ is $\sigma_N$-Lagrangian in $N$.

\smallskip
{\it Second proof} 

\begin{lemma}
The restriction $s_{P\times C}$ is symplectic on $P\times C$; hence $\dim(C)=\dim(P)$ and $Z\subset N$ is $\sigma_N$-Lagrangian.
\end{lemma}

\begin{proof}
We fix $x\in X$ and compute the $\sigma$-orthogonal $TX_x^{\perp}$ of $TX_x$, which is, by assumption, $C\times TZ_z$ if $z$ is the $Z$-component of $x$. This orthogonal consists of the quadruples $(d,c,n,p)\in D\times C\times N\times P$ such that $\sigma((d,c,n,p),(d',c',n'))=0$ for all $(d',c',n')\in D\times C\times TZ_z$. This value of $\sigma$ is equal to $\sigma_D(d,d')+\sigma_N(n,n')+s(p,(d',c',n'))$. By the linearity of this expression, for $(d,c,n,p)$ fixed, in $(d',c',n')$, this vanishing is equivalent to the following system of equations:
\begin{enumerate}
\item $s(p,c')=0$ for all $c'\in C$, 
\item $\sigma_D(d,d')+s(p,d')=0$ for all $d'\in D$,  
\item $\sigma_N(n,n')+s(p,n')=0$ for all $n'\in TZ_z$.
\end{enumerate}

Assume the first equation had a non-zero solution $p$. Since $\sigma_N$ and $\sigma_D$ are symplectic, this solution $p$ could be uniquely lifted to a solution $d$ (resp.\ $n$) of the second (resp.\ third) equation, and we would have elements $(d,c,n,p\neq 0)$, with $c\in C$ arbitrary, in $TX_x^{\perp}$. We have a contradiction since $TX^{\perp}\subset TX\subset Q$. The only solution $p$ to the first equation is thus $p=0$. The map from $P$ to the dual of $TC$ induced by $s_{P\times C}$ is therefore injective, and so $\dim(P)\leq \dim(C)$. Since moreover $s(c, c')=0$ for $c,c'\in C$, the kernel of $s_{P\times C}$ is zero; that is, $s_{P\times C}$ is symplectic. As $C$ is isotropic for $s_{P\times C}$, we also have $\dim(C)\leq \dim(P)$; hence $\dim(C)=\dim(P)$.

Since $p=0$ if $(d,c,n,p)\in TX_x^{\perp}$, from the second and third equations, we get $d=0$, and $n$ is any element of the $\sigma_N$-orthogonal of $TZ_z$, $c\in C$ being arbitrary. We thus also get $C\times TZ_z=TX_x^{\perp}=C\times TZ_z^{\perp}$, where the first equality follows from the characteristic fibration and $TZ_z^{\perp}$ is the $\sigma_N$-orthogonal. We deduce that $TZ_z=TZ_z^{\perp}$, so that $Z$ is Lagrangian.
\end{proof}

Conversely, if we are given any $Z$, $N$, $D$, $C$, $P$, $\sigma_N$, $\sigma_D$, $\sigma_P,s$ with 
\begin{enumerate}
\item $\sigma_N$, $\sigma_D$ symplectic; 
\item $Z\subset N$, $Z$ of general type generating $N$ and $\sigma_N$-Lagrangian (so that $\dim(N)$ is even); 
\item $s(p,c)$ symplectic on $P\times C$ and $C$, $P$ maximal isotropic (so that $\dim(C)=\dim(P))$; 
\end{enumerate}
then $X:=D\times C\times Z$ is $\sigma$-coisotropic, with characteristic fibration $f\colon X\to D$. This is easy to check, using the preceding arguments.

\section{The irreducible hyperk\"ahler case}\label{SIHS}

\subsection{Some known examples of Lagrangian submanifolds}

Lagrangian submanifolds of IHS manifolds have been studied to some extent, especially when the dimension of the ambient IHS manifold is $4$. Many examples are exhibited in \cite{Voisin}. Voisin also shows that the codimension, in the component of the deformation space of an IHS manifold $M$, of the locus where the Lagrangian submanifold $X\subset M$ deforms together with $M$ is equal to the rank of the restriction map $H^2(M,\Q)\to H^2(X, \Q)$. In particular, a Lagrangian $\P^n\subset M$ deforms together with $M$ in codimension $1$.

For instance, the Hilbert square of a K3 surface $S$ of genus $2$ contains a Lagrangian $\P^2$; indeed, $S$ is a double covering of a plane, and the fibres of the covering define a plane in the Hilbert square. In the same spirit, an Enriques surface $S'$ has a K3 covering $S$, and so one obtains an embedding of $S'$ into $S^{[2]}$ with Lagrangian image (because $S'$ does not carry holomorphic $2$-forms).

Within $M$, a Lagrangian $\P^n$ is isolated. But there are also examples of dominating families of Lagrangian submanifolds. It is easy to see that any Lagrangian torus deforms in a covering family (and in fact is a fibre of a Lagrangian fibration, see \cite{Amerik} for dimension $4$ and \cite{HW} in general). One can also have a dominating family of Lagrangian submanifolds of general type.\footnote{E.~Macri informed us of a conjecture stating that actually every projective IHS manifold is covered by Lagrangian submanifolds.}  Indeed, the variety of lines $M=F(V)$ of a cubic fourfold $V$ is an IHS fourfold of K3 type, and the surface of lines $S_Y$ contained in a hyperplane section $Y$ of $V$ is Lagrangian. This results from the fact that $[\sigma]$ comes from a class $\eta\in H^{3,1}(V)$ via the ``Abel--Jacobi map'', that is, the composition of the pullback to the universal family of lines on $V$ with subsequent projection to $F(V)$ (see \textit{e.g.}\ \cite{AV}). One also computes that these surfaces are of general type, using the fact that these are zero-loci of sections of the restriction of the universal bundle from the Grassmannian to $M$. Notice that the surfaces $S_Y$ cover $F(V)$.

Another source of examples is described in \cite{Beauville}: the fixed locus of an antisymplectic involution of a holomorphic symplectic manifold is Lagrangian. This applies for instance to the double EPW sextic (a double covering of a sextic of a certain type in $\P^5$ which turns out to be an IHS fourfold of K3 type, see \textit{e.g.} \cite{OG}). The fixed locus of the covering involution is a smooth Lagrangian surface $\Sigma$ of general type which does not move in a family. However, E. Macri communicated to us that there is a covering family of Lagrangian surfaces in the cohomology class $2[\Sigma]$.

In conclusion, on IHS manifolds we know a considerable variety of Lagrangian submanifolds, with negative, zero, torsion or positive canonical class.
These submanifolds can move or be fixed. The situation is very different for Abelian varieties. See Section~\ref{SAV}.

\subsection{Lagrangian fibrations}

Let $f\colon M\to B$ be a Lagrangian fibration on an IHS manifold $M$, with discriminant hypersurface $\Delta\subset B$. Let $X$ be a ``vertical'' smooth subvariety; that is, $X=f^{-1}(S)$ for some $S\subsetneq B$.  Assume that $X$ is algebraically coisotropic. If Question~\ref{q1} has a positive answer, then $X$ must be a smooth fibre of $f$.  Indeed, a stronger statement holds under the following hypothesis (this is essentially due to Abugaliev \cite{Ab1}).

\begin{proposition}\label{renat}
In the notation above, assume that $\pi_1(S-(S\cap\Delta))$ surjects onto $\pi_1(B-\Delta)$. If\, $X=f^{-1}(S)$ is a vertical submanifold, then the leaves of the characteristic foliation of\, $X$ are Zariski dense in the fibres of $f$. In particular, $X$ cannot be algebraically coisotropic unless $S$ is a point.
\end{proposition}

\begin{proof}
We briefly sketch the proof and refer to \cite{Ab1} for details (he treats the case $\codim(X)=1$, but the general one is analogous). It is easy to see that the leaves are tangent to the fibres of $f$ and the closures of the leaves in the smooth fibres are subtori. When these subtori are proper, this gives a fibration in subtori on a smooth fibre $F\subset X$, and this family of subtori is invariant under the monodromy $\pi_1(S-(S\cap\Delta))$. By a remark of Oguiso (who combined results by Voisin and Matsushita) from \cite{Ogu}, the image of the restriction map $H^2(M,\Q)\to H^2(F,\Q)$ is $1$-dimensional; hence, up to proportionality, only one class in $H^2(F,\Q)$ can be invariant under $\pi_1(B-\Delta)$. This is the restriction of an ample class on $M$, itself ample. On the other hand, a fibration of $F$ into proper subtori as above provides a nef non-ample class in $H^2(F,\Q)$, invariant under $\pi_1(S-(S\cap\Delta))$ (the pullback of an ample class on the base). Therefore, the closure of a general leaf must be the whole fibre of $f$.
\end{proof}

\subsection{Nefness and non-uniruledness}\label{r2}
A recent observation by Abugaliev and Pereira shows that a smooth non-nef hypersurface $D$ in a holomorphic symplectic $M$ is uniruled (see \cite[Theorem~8.4]{Ab2}); if $M$ is an IHS manifold, the Beauville--Bogomolov square of such a $D$ is negative (see \cite[Corollary~8.5]{Ab2}). In general, for a submanifold $X\subset M$ of higher codimension, the non-nefness of $K_X$ implies the existence of rational curves by bend-and-break. It would be interesting to see whether for algebraically coisotropic $X$, the non-nefness of $K_X$ implies uniruledness as well.

\section{Lagrangian submanifolds on Abelian varieties}\label{SAV}

There are obvious ``linear'' examples of Lagrangian submanifolds (subtori), as well as other trivial ones (curves in surfaces and products of such). In addition to Schoen's Lagrangian surfaces in certain simple Abelian fourfolds (see \cite{S}), which we have mentioned in the introduction, some non-trivial examples have been constructed using correspondences between K3 surfaces (see \cite{Bo}) or Galois closures of certain coverings (see \cite{BPS}). In these examples the Abelian variety is not simple. However, in \cite{BPS} the authors show that the square of a sufficiently general $(1,2)$-polarized Abelian surface admits a Lagrangian surface of general type, not fibered in curves (see \cite[Theorem 0.2]{BPS}). The following problem is interesting but probably difficult and out of scope of this note.

\begin{problem}
Find examples (if any) of Lagrangian submanifolds in simple Abelian varieties of dimension at least $6$.
\end{problem}

\subsection{Hodge-general Abelian varieties}

The purpose of this section is to remark that a sufficiently general Abelian variety $M$ does not have Lagrangian submanifolds. The genericity condition is stronger than simpleness and is of Hodge-theoretic nature  (see Definition~\ref{dgeneral}).

Notice the sharp contrast with the case of projective IHS manifolds. Indeed, we know that in some maximal families of projective IHS, every manifold is covered by Lagrangian submanifolds. The reason is as follows: on a Lagrangian submanifold of a sufficiently general manifold (Abelian or IHS), all holomorphic 2-forms must vanish, and on an Abelian variety, there are too many $2$-forms for this (see Corollary~\ref{chodge} and its proof).

We now give the details (see \textit{e.g.} \cite[Section~17]{BL} for more).

Let $M$ be an Abelian variety of dimension $g$ equipped with a polarization and $V$ its first cohomology $H^1(X, \Q)$ equipped with its polarized Hodge structure. We view the polarization as a skew form $\phi$ on $V$. There is a natural representation $h\colon  S^1\to \SL(V_\R)$ associated to the Hodge structure on $V$; here $S^1$ is the unit circle acting with weight $1$ on $H^{1,0}$ and $-1$ on $H^{0,1}$.

\begin{definition}\label{hodgegroup}
The Hodge group $\Hg(M)$ is the smallest algebraic subgroup of $\SL(V)$ defined over $\Q$ such that its group of $\R$-points contains $h(S^1)$.
\end{definition}

The Hodge group is always an algebraic subgroup of $\Sp(V, \phi))$ (see \cite[Proposition 17.3.2]{BL}). We say that $M$ is Hodge general when $\Hg(M)$ is as large as possible.

\begin{definition}\label{dgeneral}
A polarized Abelian variety $M$ is Hodge general if  $\Hg(M)=\Sp(V, \phi)$.
\end{definition}

By \cite[Proposition 17.4.2]{BL}, $M$ is Hodge general outside of a countable union of strict analytic subsets of the Siegel upper half-plane parametrizing $n$-dimensional Abelian varieties with a polarization $H$ of a given type (see \cite[Section~8.1]{BL}; this coarse version is sufficient for our purpose).  Moreover, the Jacobian of a sufficiently general curve is Hodge general. Many explicit examples of Hodge-general Jacobians, even for particular types of curves (\textit{e.g.}\ hyperelliptic curves) have been constructed by Zarhin (see for instance \cite{Zar}).

A Hodge-general Abelian variety is simple, and its Picard number is $1$. This immediately follows from the following more precise statement.

\begin{proposition}\label{hodge}
Let $(M,H)$ be a polarized Abelian variety. Denote by $H^{2}(M,\Q)_{\prim}$ the primitive part of the cohomology, that is, the orthogonal complement to $H\in H^2(X,\Q)$ for the pairing induced by the polarization. If $(M,H)$ is Hodge general, then $\Q H$ and $H^{2}(M,\Q)_{\prim}$ are the only Hodge substructures in $H^2(M,\Q)$.
\end{proposition} 

\begin{proof}
Since $H^2(M,\Q)=\wedge^2H^1(M,\Q)$ as polarized Hodge structures, one only needs to check that $H^{2}(M,\Q)_{\prim}$ is an irreducible $\Sp(M,H)$-module. This follows for example from \cite[Section~17, p.~260]{FH}, which asserts that for the standard representation $W$ of the symplectic group, the kernel of the natural contraction map $\wedge^kW\to \wedge^{k-2}W$ is irreducible (take $k=2$). \end{proof}

\begin{corollary}\label{chodge}
Let $M$ be a Hodge-general Abelian variety $($for some polarization $H)$ and $\sigma$ a non-zero holomorphic $2$-form on $M$. Let $X$ be an irreducible subvariety of $M$ of dimension at least $2$ and $j\colon  X'\to M$ the natural map of a desingularization of\, $X$ to $M$. Then the holomorphic 2-form $j^*\sigma$ on $X'$ is non-zero. In particular, $M$ does not contain any Lagrangian subvariety when $\dim(M)>2$.
\end{corollary}

\begin{proof} 
  The natural morphism $j^*\colon H^{2}(M,\Q)\to H^{2}(X',\Q)$ is a morphism of Hodge structures, so its kernel is a Hodge substructure of $H^{2}(M,\Q)$. If $M$ is Hodge general, there are only two Hodge substructures in $H^{2}(M,\Q)$. So in this case, if $j^*\sigma=0$, this substructure must be $H^{2}(M,\Q)_{\prim}$, and so all holomorphic 2-forms on $M$ vanish on $X'$.  This is not possible if $d=\dim(X)\geq 2$. Indeed, if $x\in X$ is a smooth point and $(z_1,\dots,z_n)$ are global coordinates on $M$ (\textit{i.e.} linear on its universal cover) such that the tangent space of $X$ at $x$ is defined by $dz_{d+1}=\dots=dz_n=0$, the global $2$-form $dz_1\wedge dz_2$ on $M$ does not vanish on $X$.
\end{proof}

\begin{remark}\label{maps}
The same argument proves that if $f\colon  Y\to M$ is a morphism to a Hodge general $M$, then $f^*\sigma\neq 0$ unless $\dim(f(Y))\leq 1$.
\end{remark}

\subsection{Products of simple Abelian varieties}\label{SPAV}

Now consider a product $M=\Pi_iM_i$ of Abelian varieties, where the factors $M_i$ are pairwise non-isogeneous, and each $M_i=S_i^{d_i}$ is a power of a simple Abelian variety $S_i$. We would like to ask the following questions:
\begin{enumerate}
\item\label{spav-1} For which $d_i$ and $\dim(S_i)$ such that $d_i \dim(S_i)$ is even can $M_i$ contain a Lagrangian subvariety (for some symplectic $\sigma$)? 

\item\label{spav-2} If none of the $M_i$ contains a Lagrangian subvariety, can $M$ contain some?
\end{enumerate}

As a variant, one may assume $S_i$ to be Hodge general (for a certain polarization) instead of just simple.

Examples for~\eqref{spav-1} indeed exist (see \cite{BPS}) if $\dim(S_i)=2$ and $d_i=2$, also with Hodge general $S_i$ equipped a $(1,2)$-polarization. It is shown in \cite{BPS} that such Abelian surfaces admit a map $\gamma\colon S\dasharrow F$ of degree $3$ to a Hirzebruch surface $F$. The component $T$ of $S\times_FS\subset S\times S$ residual to the diagonal is smooth and Lagrangian for a suitable symplectic $2$-form on $S\times S$.

Let us note that $T$ can also be described as arising from a rational (hence Lagrangian) surface $F'$ in the generalized Kummer IHS $K^{[2]}(S)$. Indeed,  let us consider the Zariski closure $T'$ in $S^3$ of the set of pairwise distinct triples $(p,q,r)$ which form a fibre of $\gamma\colon S\dasharrow F$. By definition, $T'$ is invariant by the symmetric group $S_3$ acting on $S^3$ by permutation of the factors. Thus $T'/S_3\cong F'\subset S^3/S_3$ is Lagrangian in $S^3/S_3$, and so is $T'$ in $S^3$. The image of $T'/S_3$ is contained in a fibre of $S^3/S_3\to S$ deduced from the addition on $S$: indeed, since $F$ is rational, the sum in $S$ of such a triple is constant.  Since $K^{[2]}(S)$ is birational to a fibre of the Albanese map of $S^3/S_3$, we get the claim.

\subsection{Relation to fundamental groups}

This is also related to the questions about the fundamental groups of $d$-dimensional submanifolds $X\subset M$ when $M$ is a $2d$-dimensional simple Abelian variety with $d\geq 2$.

Recall that if $a_X\colon X\to A_X$ is the Albanese map of a complex projective manifold, it induces a natural map $H^2(A_X,\C)=\wedge^2H^1(X,\C)\to H^2(X,\C)$. The kernel of this map describes the nilpotent completion\footnote{Said differently, this is the tower of torsion-free nilpotent quotients of $\pi_1(X)$.} of $\pi_1(X)$. Examples of $X\subset M$ with torsion-free nilpotent but non-Abelian $\pi_1(X)$ have been produced this way (see \cite{SV, Ca95}). These examples all have $A_X$ not simple. Moreover, the kernels are then contained in $H^{1,1}(M)$. Abelian fourfolds $M$ containing an $X$ such that there is a non-zero holomorphic $2$-form in this kernel and such that $X$ is not fibered over a curve of general type,\footnote{This property is remarkable because $X$ admits a fibration over a curve of genus $g \geq 2$ if and only if a decomposable $2$-form on $A_X$ vanishes on the image of $X$, by the Castelnuovo--de Franchis theorem.}
are constructed in \cite{Bo,BPS,S}. The following is suggested by \cite[Remarque~1.5]{Ca95}.

\begin{corollary}\label{c'hodge}
Let $M$ be a general Abelian variety, and let $h\colon X\to M$ be a holomorphic map. If $\dim(h(X))\geq 2$, $h^*\colon H^2(M,\C)\to H^2(X,\C)$ is injective, and either
\begin{enumerate}
\item\label{c'hodge-1} $q(X)=\dim(M)$, and any torsion-free nilpotent quotient of $\pi_1(X)$ is equal to $H_1(X,\Z)/\Torsion$; or 
\item\label{c'hodge-2}  $q(X)>\dim(M)$, and $A_X$ is not simple.
\end{enumerate}
\end{corollary}

\begin{proof}
  Notice that $h(X)$ generates $M$ since $M$ is simple. The injectivity claim is Remark~\ref{maps}. Assume that $q(X)=\dim(M)$. We then have an isogeny: $u\colon A_X\to M$ such that $h=u\circ a_X$; we thus get the injectivity of $a_X^*\colon H^2(A_X,\C)\to H^2(X,\C)$, and so the conclusion in case~\eqref{c'hodge-1}, by \cite[Th\'eor\`eme 3.10]{Ca95}. In case~\eqref{c'hodge-2}, we still have a factorization $h=u\circ a_X$ as before, but then the kernel of $u\colon A_X\to M$ is a positive-dimensional Abelian subvariety of $A_X$, which is therefore not simple.
\end{proof}

More generally, we ask the following. 

\begin{question}\label{qab}
  Let $M$ be a simple Abelian variety of even dimension $2n>2$. Are there restrictions on the fundamental groups of its smooth submanifolds $Y$ of dimension $n$?
\end{question}

When $\dim(Y)>n$, the fundamental group of $Y$ is Abelian (isomorphic to $\pi_1(A_Y)$), by \cite{So}. When $\dim(Y)<n$, there does not seem to be any restriction on $\pi_1(Y)$ (take the inverse image of a suitable $Z$ by a finite projection $M\to \mathbb{P}^{2n}$). In middle dimension, no example of unusual $\pi_1(Y)$ seems to be known if $M=A_Y$ is simple,  possibly due to the fact that most of them are not computed. If $M$ is simple, the Sommese--Van de Ven construction, based on the sporadic Horrocks--Mumford Abelian surfaces, gives examples in dimension $4$, but with $M\neq A_Y$.

\section{A non-projective example}\label{SNP}

\begin{proposition}\label{pnot root}
  There exist $2$-dimensional $($non-projective\/$)$ complex tori $T$ admitting an automorphism $g$ such that $g^*(\sigma)=\lambda. \sigma$
  for any symplectic $2$-form $\sigma$ on $T$, with $\lambda$ a complex number not a root of unity.
\end{proposition}

\begin{remark}
We have $\vert\lambda\vert=1$, and $T$ cannot be projective; see \cite[Section~14]{Ueno}. 
\end{remark}

\begin{corollary}\label{cnotroot}
Let $T, \lambda$ be as in Proposition~\ref{pnot root} above, and let $S$ be the smooth Kummer surface associated to $T$. Let $s$ be a symplectic form on either $T$ or $S$. On $T\times T$ and $S\times S$, there exist smooth Lagrangian surfaces for the symplectic form $(s,\lambda s)$.
\end{corollary}

\begin{proof}
  Take the graph of $g$ in $T\times T$ and its image in $S\times S$.
\end{proof}

\begin{remark}
  We do not know if there exist an Abelian surface $T$ and a Lagrangian surface $S\subset T\times T$ for a symplectic form $(s,\lambda s)$ with $\lambda$ of infinite order in $\C^*$.
\end{remark} 

\begin{proof}[Proof of Proposition~\ref{pnot root}]
The construction follows closely the one due to Iitaka in dimension $3$, as exposed in \cite[Remark~14.6, p.~179]{Ueno}. 

Let $P(X):=X^4+X+1$; it is irreducible over $\Q$ by reduction modulo $2$. The roots $a,\bar{a},b,\bar{b}$ of $P(X)$ in $\C$ are all distinct and non-real. We define $T=\C^2/R$ as the quotient of $\C^2$ with $\C$-basis $(e,e')$ by the lattice $R$ generated over $\Z$ by the four elements $R_i:=a^i.e+b^i.e'$, $i=0,1,2,3$. That $R$ is indeed a lattice follows from the fact that otherwise, there would exist a non-zero polynomial $Q(X)\in \R[X]$  of degree less that $4$ vanishing on $a$ and $b$. Thus $Q$ had to vanish on their conjugates and so be divisible by $P(X)$, so that $\deg(Q)\geq 4$, giving a contradiction. 
 
 The torus $T$ has an automorphism $g$ acting by multiplication by $a$ on $e$ and by $b$ on $e'$, and so $g$ acts on any symplectic form $\sigma$ on $T$ by multiplication by $\lambda:=ab$. We have $\vert ab\vert=1$ since the constant coefficient of $P(X)$ is $1=\vert ab\vert^2$.  It remains to see that $\lambda$ is not a root of unity, or equivalently that $\lambda$ has a Galois conjugate not of modulus $1$. It is a standard exercise to check that the Galois group $G$ of $P(X)$ is $S_4$. There is thus an element of $G$ which fixes $a$ and sends $b$ to $\bar{a}$. Thus $ab$ is $G$-conjugate to $a\bar{a}$. If $\vert a\vert^2 =1$, then $ab=1$ since it is a $G$-conjugate of $a\bar{a}$, and so $\bar{a}=b$. This gives a contradiction.
\end{proof}

\begin{remark}
The torus $T$ just constructed is simple of algebraic dimension $0$. Indeed, since $T$ is not projective, if its algebraic dimension were $1$ (or equivalently, if it were not simple), it would have a quotient fibration $f\colon T\to B=T/E$ for elliptic curves $E$, $B$. This fibration would be preserved by $\Aut(T)$, which is a finite extension of $\Aut^0(T)$ since $\Aut(E)$ and $\Aut(B)$ have finitely many components. 
\end{remark}

%%%%%%%%%%%%%%%%%%%%%
% References
%%%%%%%%%%%%%%%%%%%%%

\end{document}